\documentclass[11pt,reqno]{amsart}

\usepackage{amssymb, amsmath, amsthm}
\usepackage{hyperref}
\usepackage[alphabetic,lite]{amsrefs}
\usepackage{verbatim}
\usepackage{amscd}   
\usepackage[all]{xy} 
\usepackage{youngtab} 
\usepackage{young} 
\usepackage{ytableau}
\usepackage{tikz}
\usepackage{mathrsfs}
\usepackage{cases}
\usepackage{array}
\usepackage{tabu}
\usepackage{calligra,mathrsfs}
\usepackage{caption}

\newcommand{\defi}[1]{{\upshape\sffamily #1}}
\DeclareMathOperator{\ShHom}{\mathscr{H}\text{\kern -3pt {\calligra\large om}}\,}

\renewcommand{\b}{\beta}

\newcommand{\bw}{\bigwedge}

\newcommand{\K}{\bb{K}}
\renewcommand{\ll}{\lambda}
\newcommand{\LL}{\mathbb{L}}

\newcommand{\oo}{\otimes}

\renewcommand{\SS}{\mathbb{S}}

\newcommand{\zt}{\mathbb{Z}/2\mathbb{Z}}

\newcommand{\x}{\underline{x}}

\newcommand{\Ext}{\operatorname{Ext}}

\newcommand{\GL}{\operatorname{GL}}
\newcommand{\g}{\mathfrak{g}}
\newcommand{\gl}{\mathfrak{gl}}
\newcommand{\Hom}{\operatorname{Hom}}
\newcommand{\HS}{\operatorname{HS}}
\newcommand{\Ind}{\operatorname{Ind}}

\newcommand{\p}{\mathfrak{p}}

\newcommand{\Sym}{\operatorname{Sym}}
\newcommand{\Tor}{\operatorname{Tor}}

\newcommand{\coker}{\operatorname{coker}}

\newcommand{\dom}{\operatorname{dom}}

\newcommand{\id}{\operatorname{id}}

\newcommand{\reg}{\operatorname{reg}}

\newcommand{\supp}{\operatorname{supp}}

\newcommand{\bb}[1]{\mathbb{#1}}

\newcommand{\mc}[1]{\mathcal{#1}}
\newcommand{\mf}[1]{\mathfrak{#1}}

\newcommand{\tl}[1]{\tilde{#1}}
\newcommand{\ul}[1]{\underline{#1}}

\def\lra{\longrightarrow}

\newtheorem{theorem}{Theorem}[section]
\newtheorem*{theorem*}{Theorem}
\newtheorem*{problem*}{Problem}
\newtheorem{lemma}[theorem]{Lemma}
\newtheorem{conjecture}[theorem]{Conjecture}

\newtheorem*{corollary*}{Corollary}

\newtheorem*{main-thm*}{Main Theorem}
\newtheorem*{linear-resolutions*}{Theorem on Linear Resolutions}
\newtheorem*{regularity-powers*}{Theorem on Regularity}
\newtheorem*{injectivity-Ext*}{Theorem on Injectivity of Maps of Ext Modules}
\newtheorem*{Kodaira*}{Kodaira Vanishing for Determinantal Thickenings}

\theoremstyle{definition}

\newtheorem*{definition*}{Definition}
\newtheorem{example}[theorem]{Example}

\theoremstyle{remark}

\newtheorem*{remark*}{Remark}

\numberwithin{equation}{section}

\newdimen\unitsize
\newif\ifnumbered
\newcount\xx
\newcommand\rawpath[1]
{%
\xx=0
\coordinate[vertex] (current) at (0,0);
\foreach\x in {#1}
{
\global\advance\xx by 1
\if\x +%
\def\st{+1}
\else
\if\x 2%
\def\st{+1}
\else
\def\st{-1}
\fi
\fi
\ifnumbered
\draw[edge] (current) -- node[below,green!60!black] {$\scriptstyle \x$} ++(-1-\st,1-\st) coordinate[vertex] (current);
\else
\draw[edge] (current) -- ++(-1-\st,1-\st) coordinate[vertex] (current);
\fi
};
\ifnumbered
\draw[dotted] (0,0) -- (current |- 0,0);
\foreach\x in {1,...,\the\xx}
\path (\x,0)+(-0.5,0) node[below] {\x};
\fi
}

\begin{document}

\title[Syzygies and representations of $\gl(m|n)$]{Syzygies of determinantal thickenings and representations of the general linear Lie superalgebra}

\author{Claudiu Raicu}
\address{Department of Mathematics, University of Notre Dame, 255 Hurley, Notre Dame, IN 46556\newline
\indent Institute of Mathematics ``Simion Stoilow'' of the Romanian Academy}
\email{craicu@nd.edu}

\author{Jerzy Weyman}
\address{Department of Mathematics, University of Connecticut, Storrs, CT 06269}
\email{jerzy.weyman@uconn.edu}

\subjclass[2010]{Primary 13D02, 14M12, 17B10}

\date{\today}

\keywords{Determinantal thickenings, syzygies, BGG correspondence, general linear Lie superalgebra, Kac modules, Dyck paths}

\begin{abstract} 
 We let $S=\bb{C}[x_{i,j}]$ denote the ring of polynomial functions on the space of $m\times n$ matrices, and consider the action of the group $\GL=\GL_m\times\GL_n$ via row and column operations on the matrix entries. For a $\GL$-invariant ideal $I\subseteq S$ we show that the linear strands of its minimal free resolution translate via the BGG correspondence to modules over the general linear Lie superalgebra $\gl(m|n)$. When $I=I_{\ll}$ is the ideal generated by the $\GL$-orbit of a highest weight vector of weight $\ll$, we give a conjectural description of the classes of these $\gl(m|n)$-modules in the Grothendieck group, and prove that our prediction is correct for the first strand of the minimal free resolution.
\end{abstract}

\maketitle

\section{Introduction}\label{sec:intro}

We consider the vector space $\bb{C}^{m\times n}$ of $m\times n$ complex matrices ($m\geq n$) and let $S=\bb{C}[x_{i,j}]$ denote its coordinate ring. The group $\GL=\GL_m(\bb{C})\times\GL_n(\bb{C})$ acts on $\bb{C}^{m\times n}$ via row and column operations, making $S$ into a $\GL$-representation whose decomposition into irreducible representations is governed by Cauchy's formula: if we write $\bb{N}^n_{\dom}$ for the set of partitions with at most $n$ parts (i.e. dominant weights in $\bb{Z}^n$ with non-negative entries) and write $\SS_{\ll}$ for the \defi{Schur functor} associated to a partition $\ll$ then we have using \cite[Corollary~2.3.3]{weyman} that
\begin{equation}\label{eq:Cauchy}
S = \bigoplus_{\ll\in\bb{N}^n_{\dom}} \SS_{\ll}\bb{C}^m \oo \SS_{\ll}\bb{C}^n.
\end{equation}
When $I\subseteq S$ is a $\GL$-invariant ideal, the \defi{syzygy modules} $\Tor_i^S(I,\bb{C})$ are naturally representations of $\GL$, but their explicit description is known only in special cases \cites{lascoux,akin-buchsbaum-weyman,pragacz-weyman,raicu-weyman-syzygies}. By contrast, $\Ext^i_S(I,S)$ can be described for every $\GL$-invariant ideal $I\subseteq S$ as explained in \cite{raicu-regularity}. A special class of $\GL$-invariant ideals consists of the ones generated by a single summand $\SS_{\ll}\bb{C}^m \oo \SS_{\ll}\bb{C}^n$ in (\ref{eq:Cauchy}), and are denoted by $I_{\ll}$: one can think of them as \defi{principal $\GL$-invariant ideals}, in the sense that they are generated by the $\GL$-orbit of a single highest weight vector. The goal of this article is to propose a conjectural description of $\Tor_i^S(I_{\ll},\bb{C})$ for an arbitrary partition $\ll$, and to give supporting evidence for our conjecture.

To formulate our conjecture we re-express the problem of computing syzygies into one about modules over the exterior algebra via the BGG correspondence (described in Section~\ref{subsec:BGG}). We then relate this to the representation theory of the general linear Lie superalgebra $\gl(m|n)$ (discussed in Section~\ref{subsec:glmn}), and prove the following (see Theorem~\ref{thm:strands=glmn} for a more precise statement).

\begin{theorem*}
 The linear strands of the minimal free resolution of a $\GL$-invariant ideal translate via the BGG correspondence to finite length $\gl(m|n)$-modules.
\end{theorem*}

For the ideals of minors of the generic matrix this follows from \cite{pragacz-weyman}, and was given an alternative proof in \cite{sam} who also treats the case of symmetric and skew-symmetric matrices. For principal $\GL$-invariant ideals $I_{\ll}$ where $\ll$ is a rectangular partition, the theorem is implicit in \cite{raicu-weyman-syzygies}. Equipped with this structural result, we analyze the situation of the ideals $I_{\ll}$ for a general $\ll$. In Conjecture~\ref{conj:main} we propose an explicit formula for the class in the Grothendieck group of $\gl(m|n)$-representations of the modules encoding the linear strands of the minimal resolution of $I_{\ll}$. Our description uses the combinatorics of Dyck paths, and consists of a modification of the combinatorial rules describing type A parabolic Kazhdan--Lusztig polynomials. Using the explicit description of the submodule lattice of Kac modules from \cite{su-zhang}, we verify in Theorem~\ref{thm:linear-strand} that our conjecture correctly predicts the first linear strand of the minimal resolution of any $I_{\ll}$.

The article is organized as follows. In Section~\ref{sec:prelim} we give some background on the combinatorics of partitions and Dyck paths, discuss some basic aspects of the representation theory of general linear Lie algebras and superalgebras, and recall the statement of the BGG correspondence. In Section~\ref{sec:strands=glmn} we prove that the linear strands of $\GL$-equivariant $S$-modules that admit a decomposition analogous to (\ref{eq:Cauchy}) have the structure of $\gl(m|n)$-modules. In Section~\ref{sec:conjecture} we present a conjectural description of the syzygies of the ideals $I_{\ll}$, and in Section~\ref{sec:evidence} we offer some supporting evidence for our conjecture.

\section{Preliminaries}\label{sec:prelim}

\subsection{Partitions and Dyck paths}\label{subsec:parts+Dyck}

We write $\bb{N}^n_{\dom}$ for the set of \defi{partitions} with at most $n$ parts (or \defi{dominant weights} with non-negative integer entries). An element $\ll\in\bb{N}^n_{\dom}$ is an $n$-tuple $\ll=(\ll_1\geq\ll_2\geq\cdots\ll_n\geq 0)$. We often omit trailing zeros, for instance when we write $(4,2,2,1)$ for the partition $(4,2,2,1,0,0,0)\in\bb{N}^7_{\dom}$. We will always identify a partition with its associated \defi{Young diagram} as follows. Consider the $2$-dimensional grid induced by the inclusion of $\bb{Z}^2 \subset \bb{R}^2$, and index each box in the grid by the coordinates $(x,y)$ of its upper right corner. We identify every partition $\ll\in\bb{N}^n_{\dom}$ with the collection of boxes 
\begin{equation}\label{eq:def-ll}
\ll = \{(i,j):1\leq i\leq n, 1\leq j\leq\ll_i\}.
\end{equation}
A \defi{corner} of the partition $\ll$ is a box $(\ll_p,p)$ where $\ll_p>\ll_{p+1}$. For example, the partition $\ll=(4,2,2,1)$ has corners $(4,1)$, $(2,3)$ and $(1,4)$ and is pictured as follows:
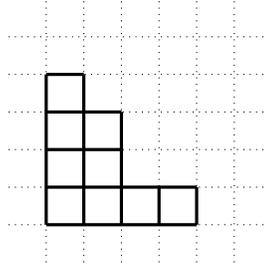
\begin{center}
\begin{tikzpicture}[x=\unitsize,y=\unitsize,baseline=0]
\tikzset{vertex/.style={}}%
\tikzset{edge/.style={very thick}}%
\draw[dotted] (0,0) -- (14,0);
\draw[dotted] (0,2) -- (14,2);
\draw[dotted] (0,4) -- (14,4);
\draw[dotted] (0,6) -- (14,6);
\draw[dotted] (0,8) -- (14,8);
\draw[dotted] (0,10) -- (14,10);
\draw[dotted] (2,-2) -- (2,12);
\draw[dotted] (4,-2) -- (4,12);
\draw[dotted] (6,-2) -- (6,12);
\draw[dotted] (8,-2) -- (8,12);
\draw[dotted] (10,-2) -- (10,12);
\draw[dotted] (12,-2) -- (12,12);
\draw[edge] (2,0) -- (10,0);
\draw[edge] (2,2) -- (10,2);
\draw[edge] (2,4) -- (6,4);
\draw[edge] (2,6) -- (6,6);
\draw[edge] (2,8) -- (4,8);
\draw[edge] (2,0) -- (2,8);
\draw[edge] (4,0) -- (4,8);
\draw[edge] (6,0) -- (6,6);
\draw[edge] (8,0) -- (8,2);
\draw[edge] (10,0) -- (10,2);
\end{tikzpicture}%
\end{center}
For partitions with repeated entries, we abbreviate a block consisting of $a$ parts of size $b$ as $(b^a)$: as an example, we write $(3,3,2,2,2,2,1)=(3^2,2^4,1)$. We will also sometimes write $a\times b$ for the partition $(b^a)$, whose associated Young diagram is a rectangle with side lengths $a$ and $b$.

A \defi{path} $P$ is a collection of boxes
\begin{equation}\label{def:path}
P = \{(x_1,y_1),(x_2,y_2),\cdots,(x_k,y_k)\}
\end{equation}
satisfying the condition that for each $i=1,\cdots,k-1$ we have that either 
\begin{equation}\label{def:steps}
(x_{i+1},y_{i+1}) = (x_i+1,y_i)\mbox{ or }(x_{i+1},y_{i+1}) = (x_i,y_i-1).
\end{equation}
The \defi{length} of the path $P$ is the number of boxes it contains, namely $k$, and is denoted by $|P|$. A \defi{corner} of $P$ is a box $(x_i,y_i)$ with $1<i<k$ and $x_{i+1}-x_{i-1}=1=y_{i+1}-y_{i-1}$. It is an \defi{inner corner} if $x_{i-1}=x_i$, and an \defi{outer corner} if $y_{i-1}=y_i$.  We say that the path $P$ is a \defi{Dyck path of level $d$} if in addition it satisfies 
\begin{itemize}
 \item $x_1+y_1 = x_k+y_k = d$.
 \item $x_i+y_i\geq d$ for every $i=1,\cdots,k$.
\end{itemize}
Note that the conditions $x_1+y_1 = x_k+y_k$ and (\ref{def:steps}) force $k$ to be odd, so the length of a Dyck path is always odd. We illustrate a path $P$ by drawing a broken line segment joining the centers of the squares it contains:

\begin{minipage}{.5\textwidth}
\centering
\begin{tikzpicture}[x=\unitsize,y=\unitsize,baseline=0]
\tikzset{vertex/.style={}}%
\tikzset{edge/.style={very thick}}%
\draw[dotted] (0,0) -- (20,0);
\draw[dotted] (0,2) -- (20,2);
\draw[dotted] (0,4) -- (20,4);
\draw[dotted] (0,6) -- (20,6);
\draw[dotted] (0,8) -- (20,8);
\draw[dotted] (0,10) -- (20,10);
\draw[dotted] (0,12) -- (20,12);
\draw[dotted] (2,-2) -- (2,14);
\draw[dotted] (4,-2) -- (4,14);
\draw[dotted] (6,-2) -- (6,14);
\draw[dotted] (8,-2) -- (8,14);
\draw[dotted] (10,-2) -- (10,14);
\draw[dotted] (12,-2) -- (12,14);
\draw[dotted] (14,-2) -- (14,14);
\draw[dotted] (16,-2) -- (16,14);
\draw[dotted] (18,-2) -- (18,14);
\draw[dotted] (3,13) -- (17,-1);
\draw[red, line width=5pt] (5,11) -- (9,11) -- (9,9) -- (11,9) -- (11,5) ;
\end{tikzpicture}
\captionsetup{labelformat=empty,justification=centering}

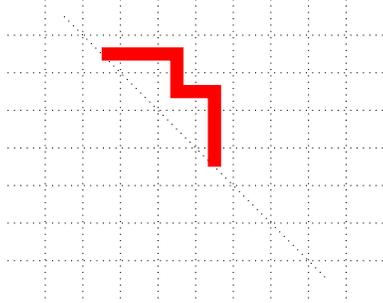
\captionof{figure}{A Dyck path of length $7$, with one inner and two outer corners}
\end{minipage}
\begin{minipage}{.5\textwidth}
\centering
\begin{tikzpicture}[x=\unitsize,y=\unitsize,baseline=0]
\tikzset{vertex/.style={}}%
\tikzset{edge/.style={very thick}}%
\draw[dotted] (0,0) -- (20,0);
\draw[dotted] (0,2) -- (20,2);
\draw[dotted] (0,4) -- (20,4);
\draw[dotted] (0,6) -- (20,6);
\draw[dotted] (0,8) -- (20,8);
\draw[dotted] (0,10) -- (20,10);
\draw[dotted] (0,12) -- (20,12);
\draw[dotted] (2,-2) -- (2,14);
\draw[dotted] (4,-2) -- (4,14);
\draw[dotted] (6,-2) -- (6,14);
\draw[dotted] (8,-2) -- (8,14);
\draw[dotted] (10,-2) -- (10,14);
\draw[dotted] (12,-2) -- (12,14);
\draw[dotted] (14,-2) -- (14,14);
\draw[dotted] (16,-2) -- (16,14);
\draw[dotted] (18,-2) -- (18,14);
\draw[dotted] (3,13) -- (17,-1);
\draw[red, line width=5pt] (5,11) -- (9,11) -- (9,9) -- (11,9) -- (11,3) -- (15,3) -- (15,1);
\end{tikzpicture}
\captionsetup{labelformat=empty,justification=centering}

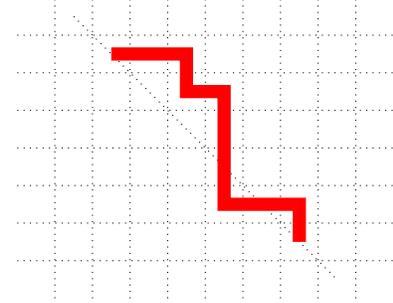
\captionof{figure}{A non-Dyck path of length $11$, with two inner and three outer corners}
\end{minipage}

An \defi{augmented Dyck path} is a pair $\tl{P}=(P,B)$ where $P$ is a Dyck path and $B$ is a set of boxes, called the \defi{bullets} in $\tl{P}$, which can be partitioned as $B = B_{head} \sqcup B_{tail}$, where (if $P$ is as in (\ref{def:path}) then)
\begin{equation}\label{eq:head-tail}
\begin{aligned}
B_{head}= \{(x_1-u,y_1), (x_1-u+1,y_1),\cdots,(x_1-1,y_1)\} &\mbox{ for some }u\geq 0,\mbox{ and} \\
B_{tail} = \{(x_k,y_k-1), (x_k,y_k-2), \cdots,(x_k,y_k-v)\} &\mbox{ for some }v\geq 0. \\
\end{aligned}
\end{equation}
The \defi{length} of $\tl{P}$ is $|\tl{P}| = |P|+u+v$, and may be an even number! To illustrate the augmented Dyck path $\tl{P}$ we draw $P$ as before, and draw small disks in the center of each of the additional $u+v$ boxes from $B$:
\begin{center}
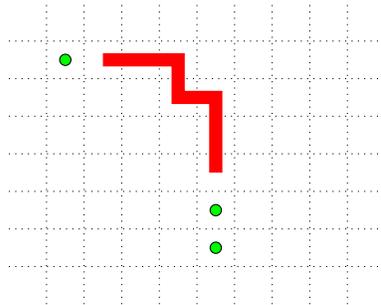

\begin{tikzpicture}[x=\unitsize,y=\unitsize,baseline=0]
\tikzset{vertex/.style={}}%
\tikzset{edge/.style={very thick}}%
\draw[dotted] (0,0) -- (20,0);
\draw[dotted] (0,2) -- (20,2);
\draw[dotted] (0,4) -- (20,4);
\draw[dotted] (0,6) -- (20,6);
\draw[dotted] (0,8) -- (20,8);
\draw[dotted] (0,10) -- (20,10);
\draw[dotted] (0,12) -- (20,12);
\draw[dotted] (2,-2) -- (2,14);
\draw[dotted] (4,-2) -- (4,14);
\draw[dotted] (6,-2) -- (6,14);
\draw[dotted] (8,-2) -- (8,14);
\draw[dotted] (10,-2) -- (10,14);
\draw[dotted] (12,-2) -- (12,14);
\draw[dotted] (14,-2) -- (14,14);
\draw[dotted] (16,-2) -- (16,14);
\draw[dotted] (18,-2) -- (18,14);
\draw[red, line width=5pt] (5,11) -- (9,11) -- (9,9) -- (11,9) -- (11,5) ;
\draw[fill=green] (11,3) circle [radius=0.3] ;
\draw[fill=green] (11,1) circle [radius=0.3] ;
\draw[fill=green] (3,11) circle [radius=0.3] ;
\end{tikzpicture}
\captionsetup{labelformat=empty}
\captionof{figure}{An augmented Dyck path of length $10$}
\end{center}

An \defi{(augmented) Dyck pattern} is a collection $\bb{D} = (D_1,D_2,\cdots,D_r;\bb{B})$ where 
\begin{itemize}
 \item each $D_i$ is a Dyck path and $\bb{B}$ is a finite set of boxes;
 \item the sets $D_1,D_2,\cdots,D_r$ and $\bb{B}$ are pairwise disjoint;
 \item $\bb{B}$ can be expressed as a union
 \begin{equation}\label{eq:bbB=union}
 \bb{B} = B_1 \cup B_2 \cup \cdots \cup B_r
 \end{equation}
 in such a way that $(D_i,B_i)$ is an augmented Dyck path for every $i=1,\cdots,r$.
\end{itemize}

\noindent Notice that we are not requiring the sets $B_i$ in (\ref{eq:bbB=union}) to be disjoint, and in particular we are not asking for the expression (\ref{eq:bbB=union}) to be unique. We write $\bb{D}=(D_1,\cdots,D_r)$ when $\bb{B}=\emptyset$. We define the \defi{support} of $\bb{D}$ by
\begin{equation}\label{eq:def-supp}
 \supp(\bb{D}) = D_1 \cup D_2 \cup \cdots \cup D_r \cup \bb{B}.
\end{equation}

If $\ll$ is a partition and $\bb{D}$ is a Dyck pattern with $\supp(\bb{D})$ disjoint from $\ll$ (when we think of $\ll$ as in (\ref{eq:def-ll})) then we let
\begin{equation}\label{eq:def-ll(D)}
 \ll(\bb{D}) = \ll \cup \supp(\bb{D}).
\end{equation}
We say that the Dyck pattern $\bb{D}$ is \defi{$\ll$-admissible} if the following conditions are satisfied:
\begin{enumerate}
\item $\ll$ is disjoint from $\supp(\bb{D})$;
\item $\ll(\bb{D})$ is (the set of boxes corresponding via (\ref{eq:def-ll}) to) a partition;
\item\label{item:cond3} For every $i\neq j$, if there exists a box $(x',y')\in D_j$ which is located directly N, E, or NE from a box $(x,y)\in D_i$ (i.e. if $(x',y')$ is one of $(x,y+1)$, $(x+1,y)$, resp. $(x+1,y+1)$), then every box located directly N, E, or NE from a box of $D_i$ must belong to $D_i$ or $D_j$.
\item\label{item:cond4} There is no bullet in $\bb{B}$ which is located directly N, E, or NE from a box in any $D_i$.
\end{enumerate}

We note that condition~(\ref{item:cond3}) above corresponds to Rule II in \cite[Section~3.1]{zinn-justin}. Below are four examples of $\ll$-admissible Dyck patterns for $\ll=(4,2,2,1)$

\begin{minipage}{.24\textwidth}
\centering
\begin{tikzpicture}[x=\unitsize,y=\unitsize,baseline=0]
\tikzset{vertex/.style={}}%
\tikzset{edge/.style={very thick}}%
\draw[dotted] (0,0) -- (14,0);
\draw[dotted] (0,2) -- (14,2);
\draw[dotted] (0,4) -- (14,4);
\draw[dotted] (0,6) -- (14,6);
\draw[dotted] (0,8) -- (14,8);
\draw[dotted] (0,10) -- (14,10);
\draw[dotted] (2,-2) -- (2,12);
\draw[dotted] (4,-2) -- (4,12);
\draw[dotted] (6,-2) -- (6,12);
\draw[dotted] (8,-2) -- (8,12);
\draw[dotted] (10,-2) -- (10,12);
\draw[dotted] (12,-2) -- (12,12);
\draw[edge] (2,0) -- (10,0);
\draw[edge] (2,2) -- (10,2);
\draw[edge] (2,4) -- (6,4);
\draw[edge] (2,6) -- (6,6);
\draw[edge] (2,8) -- (4,8);
\draw[edge] (2,0) -- (2,8);
\draw[edge] (4,0) -- (4,8);
\draw[edge] (6,0) -- (6,6);
\draw[edge] (8,0) -- (8,2);
\draw[edge] (10,0) -- (10,2);
\draw[red, line width=5pt] (5,7) -- (7,7) -- (7,5) ;
\draw[fill=green] (7,3) circle [radius=0.3] ;
\draw[red, line width=5pt] (5,9) -- (9,9) -- (9,5) ; 
\draw[fill=green] (9,3) circle [radius=0.3] ; 
\draw[fill=green] (3,9) circle [radius=0.3] ;
\end{tikzpicture}%
\end{minipage}
\begin{minipage}{.24\textwidth}
\centering
\begin{tikzpicture}[x=\unitsize,y=\unitsize,baseline=0]
\tikzset{vertex/.style={}}%
\tikzset{edge/.style={very thick}}%
\draw[dotted] (0,0) -- (14,0);
\draw[dotted] (0,2) -- (14,2);
\draw[dotted] (0,4) -- (14,4);
\draw[dotted] (0,6) -- (14,6);
\draw[dotted] (0,8) -- (14,8);
\draw[dotted] (0,10) -- (14,10);
\draw[dotted] (2,-2) -- (2,12);
\draw[dotted] (4,-2) -- (4,12);
\draw[dotted] (6,-2) -- (6,12);
\draw[dotted] (8,-2) -- (8,12);
\draw[dotted] (10,-2) -- (10,12);
\draw[dotted] (12,-2) -- (12,12);
\draw[edge] (2,0) -- (10,0);
\draw[edge] (2,2) -- (10,2);
\draw[edge] (2,4) -- (6,4);
\draw[edge] (2,6) -- (6,6);
\draw[edge] (2,8) -- (4,8);
\draw[edge] (2,0) -- (2,8);
\draw[edge] (4,0) -- (4,8);
\draw[edge] (6,0) -- (6,6);
\draw[edge] (8,0) -- (8,2);
\draw[edge] (10,0) -- (10,2);
\draw[red, line width=5pt] (5,7) -- (7,7) -- (7,5) ;
\draw[fill=green] (7,3) circle [radius=0.3] ;
\draw[red, line width=5pt] (3,9) -- (9,9) -- (9,3) ;
\end{tikzpicture}%
\end{minipage}
\begin{minipage}{.24\textwidth}
\centering
\begin{tikzpicture}[x=\unitsize,y=\unitsize,baseline=0]
\tikzset{vertex/.style={}}%
\tikzset{edge/.style={very thick}}%
\draw[dotted] (0,0) -- (14,0);
\draw[dotted] (0,2) -- (14,2);
\draw[dotted] (0,4) -- (14,4);
\draw[dotted] (0,6) -- (14,6);
\draw[dotted] (0,8) -- (14,8);
\draw[dotted] (0,10) -- (14,10);
\draw[dotted] (2,-2) -- (2,12);
\draw[dotted] (4,-2) -- (4,12);
\draw[dotted] (6,-2) -- (6,12);
\draw[dotted] (8,-2) -- (8,12);
\draw[dotted] (10,-2) -- (10,12);
\draw[dotted] (12,-2) -- (12,12);
\draw[edge] (2,0) -- (10,0);
\draw[edge] (2,2) -- (10,2);
\draw[edge] (2,4) -- (6,4);
\draw[edge] (2,6) -- (6,6);
\draw[edge] (2,8) -- (4,8);
\draw[edge] (2,0) -- (2,8);
\draw[edge] (4,0) -- (4,8);
\draw[edge] (6,0) -- (6,6);
\draw[edge] (8,0) -- (8,2);
\draw[edge] (10,0) -- (10,2);
\draw[red, line width=5pt] (5,7) -- (7,7) -- (7,5) ;
\draw[fill=green] (7,3) circle [radius=0.3] ;
\draw[red, line width=5pt] (3,9) -- (9,9) -- (9,3) -- (11,3) -- (11,1);
\end{tikzpicture}%
\end{minipage}
\begin{minipage}{.24\textwidth}
\centering
\begin{tikzpicture}[x=\unitsize,y=\unitsize,baseline=0]
\tikzset{vertex/.style={}}%
\tikzset{edge/.style={very thick}}%
\draw[dotted] (0,0) -- (14,0);
\draw[dotted] (0,2) -- (14,2);
\draw[dotted] (0,4) -- (14,4);
\draw[dotted] (0,6) -- (14,6);
\draw[dotted] (0,8) -- (14,8);
\draw[dotted] (0,10) -- (14,10);
\draw[dotted] (2,-2) -- (2,12);
\draw[dotted] (4,-2) -- (4,12);
\draw[dotted] (6,-2) -- (6,12);
\draw[dotted] (8,-2) -- (8,12);
\draw[dotted] (10,-2) -- (10,12);
\draw[dotted] (12,-2) -- (12,12);
\draw[edge] (2,0) -- (10,0);
\draw[edge] (2,2) -- (10,2);
\draw[edge] (2,4) -- (6,4);
\draw[edge] (2,6) -- (6,6);
\draw[edge] (2,8) -- (4,8);
\draw[edge] (2,0) -- (2,8);
\draw[edge] (4,0) -- (4,8);
\draw[edge] (6,0) -- (6,6);
\draw[edge] (8,0) -- (8,2);
\draw[edge] (10,0) -- (10,2);
\draw[red, line width=5pt] (7,5) -- (9,5) -- (9,3) ;
\draw[red, line width=5pt] (3,9) -- (5,9) -- (5,7) ;
\draw[red, line width=5pt] (6.6,3) -- (7.4,3) ;
\draw[red, line width=5pt] (10.6,1) -- (11.4,1) ;
\end{tikzpicture}%
\end{minipage}
and three examples of Dyck patterns that are not $\ll$-admissible

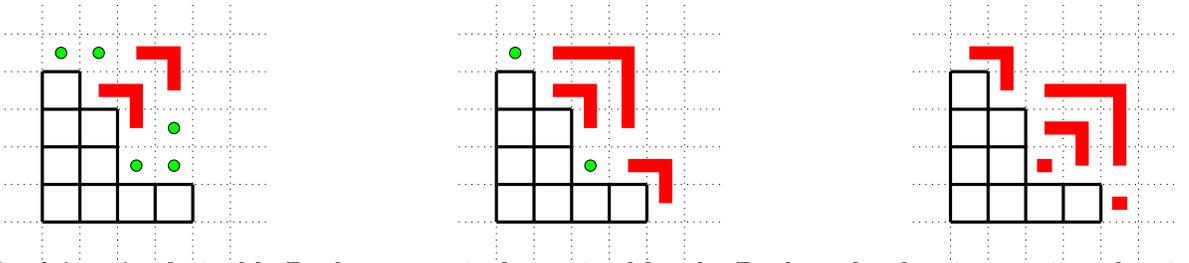
\begin{minipage}{.33\textwidth}
\centering
\begin{tikzpicture}[x=\unitsize,y=\unitsize,baseline=0]
\tikzset{vertex/.style={}}%
\tikzset{edge/.style={very thick}}%
\draw[dotted] (0,0) -- (14,0);
\draw[dotted] (0,2) -- (14,2);
\draw[dotted] (0,4) -- (14,4);
\draw[dotted] (0,6) -- (14,6);
\draw[dotted] (0,8) -- (14,8);
\draw[dotted] (0,10) -- (14,10);
\draw[dotted] (2,-2) -- (2,12);
\draw[dotted] (4,-2) -- (4,12);
\draw[dotted] (6,-2) -- (6,12);
\draw[dotted] (8,-2) -- (8,12);
\draw[dotted] (10,-2) -- (10,12);
\draw[dotted] (12,-2) -- (12,12);
\draw[edge] (2,0) -- (10,0);
\draw[edge] (2,2) -- (10,2);
\draw[edge] (2,4) -- (6,4);
\draw[edge] (2,6) -- (6,6);
\draw[edge] (2,8) -- (4,8);
\draw[edge] (2,0) -- (2,8);
\draw[edge] (4,0) -- (4,8);
\draw[edge] (6,0) -- (6,6);
\draw[edge] (8,0) -- (8,2);
\draw[edge] (10,0) -- (10,2);
\draw[red, line width=5pt] (5,7) -- (7,7) -- (7,5) ;
\draw[fill=green] (7,3) circle [radius=0.3] ;
\draw[red, line width=5pt] (7,9) -- (9,9) -- (9,7) ; 
\draw[fill=green] (9,3) circle [radius=0.3] ; 
\draw[fill=green] (3,9) circle [radius=0.3] ;
\draw[fill=green] (9,5) circle [radius=0.3] ; 
\draw[fill=green] (5,9) circle [radius=0.3] ;
\end{tikzpicture}%
\end{minipage}
\begin{minipage}{.33\textwidth}
\centering
\begin{tikzpicture}[x=\unitsize,y=\unitsize,baseline=0]
\tikzset{vertex/.style={}}%
\tikzset{edge/.style={very thick}}%
\draw[dotted] (0,0) -- (14,0);
\draw[dotted] (0,2) -- (14,2);
\draw[dotted] (0,4) -- (14,4);
\draw[dotted] (0,6) -- (14,6);
\draw[dotted] (0,8) -- (14,8);
\draw[dotted] (0,10) -- (14,10);
\draw[dotted] (2,-2) -- (2,12);
\draw[dotted] (4,-2) -- (4,12);
\draw[dotted] (6,-2) -- (6,12);
\draw[dotted] (8,-2) -- (8,12);
\draw[dotted] (10,-2) -- (10,12);
\draw[dotted] (12,-2) -- (12,12);
\draw[edge] (2,0) -- (10,0);
\draw[edge] (2,2) -- (10,2);
\draw[edge] (2,4) -- (6,4);
\draw[edge] (2,6) -- (6,6);
\draw[edge] (2,8) -- (4,8);
\draw[edge] (2,0) -- (2,8);
\draw[edge] (4,0) -- (4,8);
\draw[edge] (6,0) -- (6,6);
\draw[edge] (8,0) -- (8,2);
\draw[edge] (10,0) -- (10,2);
\draw[red, line width=5pt] (5,7) -- (7,7) -- (7,5) ;
\draw[fill=green] (7,3) circle [radius=0.3] ;
\draw[fill=green] (3,9) circle [radius=0.3] ;
\draw[red, line width=5pt] (5,9) -- (9,9) -- (9,5) ; 
\draw[red, line width=5pt] (9,3) -- (11,3) -- (11,1) ; 
\end{tikzpicture}%
\end{minipage}
\begin{minipage}{.33\textwidth}
\centering
\begin{tikzpicture}[x=\unitsize,y=\unitsize,baseline=0]
\tikzset{vertex/.style={}}%
\tikzset{edge/.style={very thick}}%
\draw[dotted] (0,0) -- (14,0);
\draw[dotted] (0,2) -- (14,2);
\draw[dotted] (0,4) -- (14,4);
\draw[dotted] (0,6) -- (14,6);
\draw[dotted] (0,8) -- (14,8);
\draw[dotted] (0,10) -- (14,10);
\draw[dotted] (2,-2) -- (2,12);
\draw[dotted] (4,-2) -- (4,12);
\draw[dotted] (6,-2) -- (6,12);
\draw[dotted] (8,-2) -- (8,12);
\draw[dotted] (10,-2) -- (10,12);
\draw[dotted] (12,-2) -- (12,12);
\draw[edge] (2,0) -- (10,0);
\draw[edge] (2,2) -- (10,2);
\draw[edge] (2,4) -- (6,4);
\draw[edge] (2,6) -- (6,6);
\draw[edge] (2,8) -- (4,8);
\draw[edge] (2,0) -- (2,8);
\draw[edge] (4,0) -- (4,8);
\draw[edge] (6,0) -- (6,6);
\draw[edge] (8,0) -- (8,2);
\draw[edge] (10,0) -- (10,2);
\draw[red, line width=5pt] (3,9) -- (5,9) -- (5,7) ;
\draw[red, line width=5pt] (7,7) -- (11,7) -- (11,3) ;
\draw[red, line width=5pt] (7,5) -- (9,5) -- (9,3) ;
\draw[red, line width=5pt] (6.6,3) -- (7.4,3) ;
\draw[red, line width=5pt] (10.6,1) -- (11.4,1) ;
\end{tikzpicture}%
\end{minipage}

For a fixed $\ll$, a $\ll$-admissible Dyck pattern is determined by the Dyck paths that it contains, that is the position of the bullets is determined by $\ll$ and the Dyck paths. Since
\[\ll(\bb{D}) = \ll \sqcup D_1 \sqcup \cdots \sqcup D_r \sqcup \bb{B}\]
this amounts to the fact that $\ll$ together with the Dyck paths in $D$ determine $\ll(\bb{D})$, which we prove next.

\begin{lemma}\label{lem:adm-pat-determines-bullets}
 Let $\bb{D}=(D_1,\cdots,D_r;\bb{B})$ be a $\ll$-admissible Dyck pattern and let $x,y\geq 1$ be positive integers. The following statements are equivalent:
 \begin{enumerate}
 \item $(x,y)\in\ll(\bb{D})$.
 \item $(x,y)\in\ll$ or there exists an index $1\leq i\leq r$ and a box $(x',y')\in D_i$ satisfying $x'\geq x$ and $y'\geq y$.
 \end{enumerate}
\end{lemma}

\begin{proof} ``(1) $\Rightarrow$ (2)": Suppose that $(x,y)\in\ll(\bb{D})\setminus\ll$. If $(x,y)\in D_i$ for some $i$ then we can take $(x',y')=(x,y)$ to get (2). If $(x,y)\in\bb{B}$ then using notation (\ref{eq:bbB=union}) we have $(x,y)\in B_i$ for some $i$, i.e. $(x,y)$ is a bullet in the augmented Dyck path $(D_i,B_i)$. Letting $(P,B)=(D_i,B_i)$ and using notation (\ref{eq:head-tail}) we get that either $(x,y)\in B_{head}$ in which case we take $(x',y')=(x_1,y_1)$, or $(x,y)\in B_{tail}$ when we take $(x',y')=(x_k,y_k)$.

``(2) $\Rightarrow$ (1)": Since $\ll(\bb{D})$ is a partition, it follows that if $(x',y')\in\ll(\bb{D})$ and $1\leq x\leq x'$, $1\leq y\leq y'$, then $(x,y)\in\ll(\bb{D})$. Since $\ll\subseteq\ll(\bb{D})$ and $D_i\subseteq\ll(\bb{D})$, the conclusion follows.
\end{proof}

We define the \defi{Dyck size} of $\bb{D}$ to be
\[d(\bb{D}) = |D_1|+|D_2|+\cdots+|D_r|\]
and the \defi{bullet size} of $\bb{D}$ to be
\[b(\bb{D}) = |\bb{B}|.\]
The \defi{(total) size} of $\bb{D}$ is $|\bb{D}| = d(\bb{D}) + b(\bb{D})$, so that $|\ll(\bb{D})| = |\ll| + |\bb{D}|$ for every $\ll$-admissible Dyck pattern $\bb{D}$.

\subsection{The general linear Lie algebra}\label{subsec:repthy}

Let $U$ be a finite dimensional complex vector space with $\dim(U)=r$, and let $\gl(U)$ the Lie algebra of endomorphisms of $U$, with the usual Lie bracket $[x,y]=xy-yx$. We write $\bb{Z}^n_{\dom}$ for the set of dominant weights $\ll\in\bb{Z}^n$ with $\ll_1\geq\ll_2\geq\cdots\geq\ll_n$, and write $\SS_{\ll}$ for the \defi{Schur functor} associated to $\ll$. We write $U^{\vee} = \Hom_{\bb{C}}(U,\bb{C})$ for the dual vector space, and $\ll^{\vee} = (-\ll_n,-\ll_{n-1},\cdots,-\ll_1)$, so that we have a natural isomorphism
\begin{equation}\label{eq:sll-dual}
\SS_{\ll}U^{\vee} \simeq \SS_{\ll^{\vee}}U.
\end{equation}
Our convention for Schur functors is so that if $\ll=(d,0^{N-1})=(d)$ for $d\geq 0$ then $\SS_{\ll}U=\Sym^d U$, and if $\ll=(1^k)$ then $\SS_{\ll}U = \bw^k U$.

There is a natural isomorphism of Lie algebras $\gl(U) \simeq \gl(U^{\vee})$ given by $\phi\mapsto -\phi^{\vee}$. A choice of basis on $U$ determines a maximal torus $\mf{t}$ of diagonal matrices inside $\gl(U)$, and a dual basis of $U^{\vee}$ with a corresponding maximal torus~$\mf{t}^{\vee}$ inside $\gl(U^{\vee})$. Because the natural identification $\gl(U) \simeq \gl(U^{\vee})$ sends $\phi\mapsto -\phi^{\vee}$, positive weights with respect to $\mf{t}$ will correspond to negative weights with respect to $\mf{t}^{\vee}$ and vice-versa (the equation (\ref{eq:sll-dual}) is an instance of this phenomenon). Based on this observation, we will choose our conventions so that we are only required to work with partitions (non-negative dominant weights) $\ll\in\bb{N}^n_{\dom}$ in the rest of the article, allowing us to take advantage of the pictorial representation described in the previous section. 

\subsection{The BGG correspondence}\label{subsec:BGG}

Throughout this article we let $V_0,V_1$ be complex vector spaces with $\dim(V_0)=m$, $\dim(V_1)=n$, and assume that $m\geq n$. We write $W_i=V_i^{\vee}$ for their vector space duals, and let $V=V_0\oo V_1$ and $W=W_0\oo W_1=V^{\vee}$. We consider the polynomial ring $S=\Sym(V)$ and the exterior algebra $E=\bw W$. Choosing dual bases on the spaces $V_i$ and $W_i$, we can identify $S = \bb{C}[x_{i,j}]$ and $E=\bb{C}\langle e_{i,j}\rangle$, where $\langle,\rangle$ indicates that the multiplication in $E$ is skew-commutative.

If $M=\bigoplus_{t\in\bb{Z}}M_t$ is a finitely generated graded $S$-module, we let $M^{\vee}$ denote its \defi{graded dual},
\[M^{\vee} = \bigoplus_{t\in\bb{Z}} \Hom_{\bb{C}}(M_t,\bb{C}) = \bigoplus_{t\in\bb{Z}}M_t^{\vee},\]
where the action of $S$ is given by $(s\cdot\phi)(m) = \phi(s\cdot m)$ for $s\in S$, $\phi\in M^{\vee}$ and $m\in M$ homogeneous elements. We associate to $M$ a complex $\tl{\bf R}(M)$ of free $E$-modules (which is a modification of the complex ${\bf R}(M)$ in \cite[Section~7E]{eisenbud-syzygies}):
\[\tl{\bf R}(M):\quad \cdots \lra E\oo M_t^{\vee} \lra E\oo M_{t-1}^{\vee}\lra \cdots \]
\[ e\oo\phi \lra \sum_{i,j} e\cdot e_{i,j} \oo x_{i,j}\cdot\phi\]
We make the convention that $E_s = \bw^s W$ lies in degree $s$, that is we grade $E$ positively with respect to the ``$W$-variables" $e_{i,j}$, or more formally we take the grading induced by the action of the $1$-dimensional torus spanned by $(\id_{W_0},\id_{W_1})$ inside $\gl(W_0)\oplus\gl(W_1)$. This is different from \cite[Section~7B]{eisenbud-syzygies} where the $W$-variables are given negative degrees, since the grading is relative to the action of the $1$-dimensional torus spanned by $(\id_{V_0},\id_{V_1})$ inside $\gl(V_0)\oplus\gl(V_1)$. With this convention we give $M^{\vee}_t$ degree $t$ (with respect to the ``$W$-variables") and the analogue of \cite[Proposition~7.21]{eisenbud-syzygies} yields
\[H_t(\tl{\bf R}(M))_{s+t} \simeq \Tor_s(\bb{C},M)^{\vee}_{s+t}.\]
The $E$-module $H_t(\tl{\bf R}(M))$ is finitely generated (and in particular a finite dimensional vector space) and it encodes (up to taking vector space duals) the $t$-th linear strand of the minimal free resolution of $M$. Furthermore, if $M$ is a $(\gl(V_0)\oplus\gl(V_1))$-equivariant $S$-module then each $H_t(\tl{\bf R}(M))$ is a $(\gl(W_0)\oplus\gl(W_1))$-equivariant $E$-module. With some more assumptions on $M$, we will see in Section~\ref{sec:strands=glmn} that $H_t(\tl{\bf R}(M))$ is a module over the general linear Lie superalgebra discussed next.

\subsection{Representations of the general linear Lie superalgebra}\label{subsec:glmn}


We let $\g=\gl(m|n)$ denote the general linear Lie superalgebra of endomorphisms of the $\zt$-graded vector space $W_0\oplus V_1$, where $W_0\simeq\bb{C}^m$ is in degree $0$, and $V_1\simeq\bb{C}^n$ is in degree $1$. As in the previous section we let $V_0=W_0^{\vee}$, $W_1=V_1^{\vee}$. We consider the $\bb{Z}$-grading on $\g$ given by
\[\g_{0} = \gl(V_0)\oplus\gl(W_1)\simeq \gl(W_0)\oplus\gl(W_1),\]
\[ \g_{-1} = \Hom_{\bb{C}}(V_0,W_1)\simeq W_0\oo W_1,\mbox{ and }\g_{1} = \Hom_{\bb{C}}(W_1,V_0)\simeq V_0\oo V_1,\]
and the Lie superbracket $[x,y] = xy- (-1)^{\deg(x)\cdot\deg(y)}yx$ for $x,y$ homogeneous elements of $\g$. Note that the superbracket restricts to a usual Lie bracket on $\g_0$, which itself is a reductive Lie algebra. We define
\[\p = \g_{0} \oplus \g_1\]
which is a subalgebra of $\g$, and observe that every $\g_0$-module $M$ can be thought of as a $\p$-module by making the action of $\g_1$ on $M$ be trivial. For every partition $\ll\in\bb{N}^n_{\dom}$ we can then take the irreducible $\g_0$-module $\SS_{\ll}W_0\oo\SS_{\ll}W_1$, think of it as a $\p$-module, and define the induced representation
\begin{equation}\label{eq:def-Kac-module}
 \K_{\ll} = \Ind_{\p}^{\g}(\SS_{\ll}W_0\oo\SS_{\ll}W_1),
\end{equation} 
which we call the \defi{Kac module of weight $\ll$}. We note that in the general theory of representations of $\gl(m|n)$ one considers more general Kac modules by inducing $\SS_{\ll}W_0\oo\SS_{\mu}W_1$ for an arbitrary pair of partitions (or more generally, dominant weights) $(\ll,\mu)$. The special case of Kac modules that we consider in (\ref{eq:def-Kac-module}) are the ones of so called \defi{maximal degree of atypicality}, and in a sense are the most interesting of the Kac modules. They lie at the opposite end of the spectrum from the \defi{typical Kac modules} (those whose degree of atypicality is $0$), which are known to be irreducible as $\g$-modules. By contrast, the modules $\K_{\ll}$ in (\ref{eq:def-Kac-module}) have a very interesting $\g$-module structure which will be discussed next.

To motivate our interest in Kac modules, we note that $\g_{-1}=W_0\oo W_1$ is an abelian Lie superalgebra concentrated in odd degree, so its universal enveloping algebra is simply $\mc{U}(\g_{-1})=\bigwedge\g_{-1}$, the exterior algebra which was denoted by $E$ in Section~\ref{subsec:BGG}. It follows that $\mc{U}(\g)$ contains $E$ as a subring, and therefore every $\g$-module inherits the structure of an $E$-module. Moreover, since $\g_0\subset\g$ is a subalgebra, any such module is also $\g_0$-equivariant. The Kac modules are in fact free as $E$-modules,
\[\K_{\ll} = E\oo(\SS_{\ll}W_0\oo\SS_{\ll}W_1),\]
and their $\g_0$-module structure can be obtained based on the Cauchy decomposition of exterior powers of a tensor product, combined with the Littlewood--Richardson rule. As a $\g$-module, $\K_{\ll}$ has a unique simple quotient, which is denote by $\LL_{\ll}$ -- it is the \defi{simple $\g$-module of weight $\ll$}. $\K_{\ll}$ is not semi-simple as a $\g$-module, but it has finite length with composition factors described as follows. We let
\begin{equation}\label{eq:K-ll-n}
\mc{K}(\ll;n) = \{ \bb{D}=(D_1,\cdots,D_r)\mbox{ a }\ll\mbox{-admissible Dyck pattern and }\ll(\bb{D})_j = 0\mbox{ for }j>n\}
\end{equation}
and stress the fact that the patterns in $\mc{K}(\ll;n)$ are not augmented, i.e. they contain no bullets, but may contain Dyck paths of length one. The composition factors of the Kac modules are encoded by parabolic versions of Kazhdan-Lusztig polynomials \cites{brundan,serganova}. Using the Dyck pattern interpretation of the parabolic Kazhdan--Lusztig polynomials based on Rule~II in \cite[Section~3.1]{zinn-justin} we get the following.

\begin{theorem}\label{thm:compos-Kac}
If we let $[M]$ denote the class of a $\g$-module $M$ in the Grothendieck group $K_0(\g)$ of finite dimensional representations of $\g$, then
\[[\K_{\ll}] = \sum_{\bb{D}\in\mc{K}(\ll;n)} [\LL_{\ll(\bb{D})}].\]
\end{theorem}

\begin{example}\label{ex:compos-Kac} Take $n=3$ and consider $\ll=(3,2)$. The Kac module $\K_{\ll}$ has $10$ simple composition factors, classified by the Dyck patterns $\bb{D}$ pictured below (and labelled by $\ll(\bb{D})$)

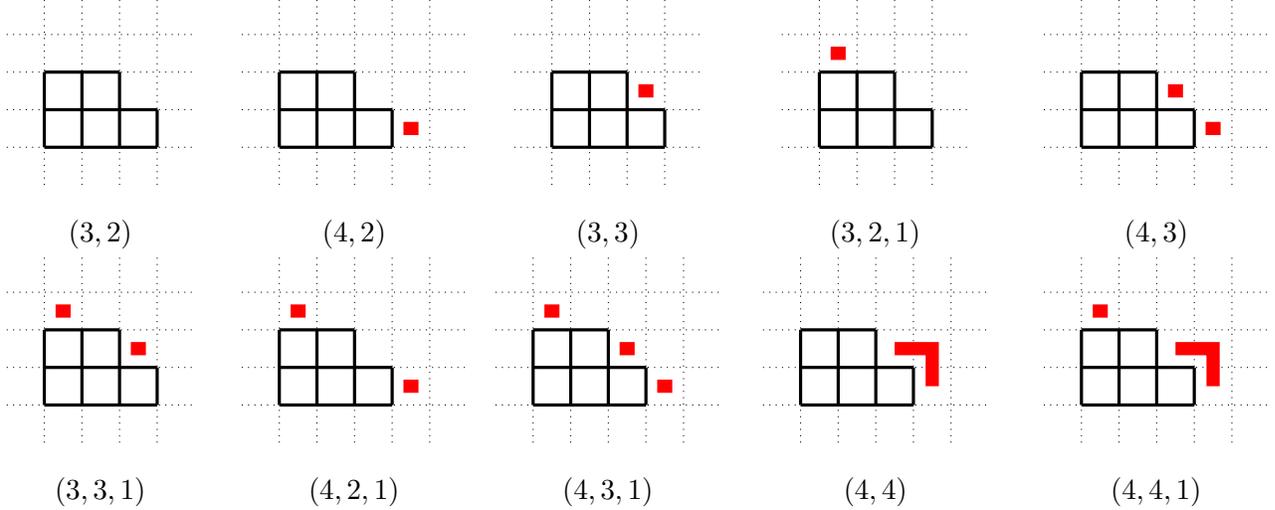
\begin{minipage}{.18\textwidth}
\centering
\begin{tikzpicture}[x=\unitsize,y=\unitsize,baseline=0]
\tikzset{vertex/.style={}}%
\tikzset{edge/.style={very thick}}%
\draw[dotted] (0,0) -- (10,0);
\draw[dotted] (0,2) -- (10,2);
\draw[dotted] (0,4) -- (10,4);
\draw[dotted] (0,6) -- (10,6);
\draw[dotted] (2,-2) -- (2,8);
\draw[dotted] (4,-2) -- (4,8);
\draw[dotted] (6,-2) -- (6,8);
\draw[dotted] (8,-2) -- (8,8);
\draw[edge] (2,0) -- (8,0);
\draw[edge] (2,2) -- (8,2);
\draw[edge] (2,4) -- (6,4);
\draw[edge] (2,0) -- (2,4);
\draw[edge] (4,0) -- (4,4);
\draw[edge] (6,0) -- (6,4);
\draw[edge] (8,0) -- (8,2);
\end{tikzpicture}%
\captionsetup{labelformat=empty}
\captionof{figure}{$(3,2)$}
\end{minipage}
\begin{minipage}{.18\textwidth}
\centering
\begin{tikzpicture}[x=\unitsize,y=\unitsize,baseline=0]
\tikzset{vertex/.style={}}%
\tikzset{edge/.style={very thick}}%
\draw[dotted] (0,0) -- (12,0);
\draw[dotted] (0,2) -- (12,2);
\draw[dotted] (0,4) -- (12,4);
\draw[dotted] (0,6) -- (12,6);
\draw[dotted] (2,-2) -- (2,8);
\draw[dotted] (4,-2) -- (4,8);
\draw[dotted] (6,-2) -- (6,8);
\draw[dotted] (8,-2) -- (8,8);
\draw[dotted] (10,-2) -- (10,8);
\draw[edge] (2,0) -- (8,0);
\draw[edge] (2,2) -- (8,2);
\draw[edge] (2,4) -- (6,4);
\draw[edge] (2,0) -- (2,4);
\draw[edge] (4,0) -- (4,4);
\draw[edge] (6,0) -- (6,4);
\draw[edge] (8,0) -- (8,2);
\draw[red, line width=5pt] (8.6,1) -- (9.4,1) ;
\end{tikzpicture}%
\captionsetup{labelformat=empty}
\captionof{figure}{$(4,2)$}
\end{minipage}
\begin{minipage}{.18\textwidth}
\centering
\begin{tikzpicture}[x=\unitsize,y=\unitsize,baseline=0]
\tikzset{vertex/.style={}}%
\tikzset{edge/.style={very thick}}%
\draw[dotted] (0,0) -- (10,0);
\draw[dotted] (0,2) -- (10,2);
\draw[dotted] (0,4) -- (10,4);
\draw[dotted] (0,6) -- (10,6);
\draw[dotted] (2,-2) -- (2,8);
\draw[dotted] (4,-2) -- (4,8);
\draw[dotted] (6,-2) -- (6,8);
\draw[dotted] (8,-2) -- (8,8);
\draw[edge] (2,0) -- (8,0);
\draw[edge] (2,2) -- (8,2);
\draw[edge] (2,4) -- (6,4);
\draw[edge] (2,0) -- (2,4);
\draw[edge] (4,0) -- (4,4);
\draw[edge] (6,0) -- (6,4);
\draw[edge] (8,0) -- (8,2);
\draw[red, line width=5pt] (6.6,3) -- (7.4,3) ;
\end{tikzpicture}%
\captionsetup{labelformat=empty}
\captionof{figure}{$(3,3)$}
\end{minipage}
\begin{minipage}{.2\textwidth}
\centering
\begin{tikzpicture}[x=\unitsize,y=\unitsize,baseline=0]
\tikzset{vertex/.style={}}%
\tikzset{edge/.style={very thick}}%
\draw[dotted] (0,0) -- (10,0);
\draw[dotted] (0,2) -- (10,2);
\draw[dotted] (0,4) -- (10,4);
\draw[dotted] (0,6) -- (10,6);
\draw[dotted] (2,-2) -- (2,8);
\draw[dotted] (4,-2) -- (4,8);
\draw[dotted] (6,-2) -- (6,8);
\draw[dotted] (8,-2) -- (8,8);
\draw[edge] (2,0) -- (8,0);
\draw[edge] (2,2) -- (8,2);
\draw[edge] (2,4) -- (6,4);
\draw[edge] (2,0) -- (2,4);
\draw[edge] (4,0) -- (4,4);
\draw[edge] (6,0) -- (6,4);
\draw[edge] (8,0) -- (8,2);
\draw[red, line width=5pt] (2.6,5) -- (3.4,5) ;
\end{tikzpicture}%
\captionsetup{labelformat=empty}
\captionof{figure}{$(3,2,1)$}
\end{minipage}
\begin{minipage}{.2\textwidth}
\centering
\begin{tikzpicture}[x=\unitsize,y=\unitsize,baseline=0]
\tikzset{vertex/.style={}}%
\tikzset{edge/.style={very thick}}%
\draw[dotted] (0,0) -- (12,0);
\draw[dotted] (0,2) -- (12,2);
\draw[dotted] (0,4) -- (12,4);
\draw[dotted] (0,6) -- (12,6);
\draw[dotted] (2,-2) -- (2,8);
\draw[dotted] (4,-2) -- (4,8);
\draw[dotted] (6,-2) -- (6,8);
\draw[dotted] (8,-2) -- (8,8);
\draw[dotted] (10,-2) -- (10,8);
\draw[edge] (2,0) -- (8,0);
\draw[edge] (2,2) -- (8,2);
\draw[edge] (2,4) -- (6,4);
\draw[edge] (2,0) -- (2,4);
\draw[edge] (4,0) -- (4,4);
\draw[edge] (6,0) -- (6,4);
\draw[edge] (8,0) -- (8,2);
\draw[red, line width=5pt] (6.6,3) -- (7.4,3) ;
\draw[red, line width=5pt] (8.6,1) -- (9.4,1) ;
\end{tikzpicture}%
\captionsetup{labelformat=empty}
\captionof{figure}{$(4,3)$}
\end{minipage}

\begin{minipage}{.18\textwidth}
\centering
\begin{tikzpicture}[x=\unitsize,y=\unitsize,baseline=0]
\tikzset{vertex/.style={}}%
\tikzset{edge/.style={very thick}}%
\draw[dotted] (0,0) -- (10,0);
\draw[dotted] (0,2) -- (10,2);
\draw[dotted] (0,4) -- (10,4);
\draw[dotted] (0,6) -- (10,6);
\draw[dotted] (2,-2) -- (2,8);
\draw[dotted] (4,-2) -- (4,8);
\draw[dotted] (6,-2) -- (6,8);
\draw[dotted] (8,-2) -- (8,8);
\draw[edge] (2,0) -- (8,0);
\draw[edge] (2,2) -- (8,2);
\draw[edge] (2,4) -- (6,4);
\draw[edge] (2,0) -- (2,4);
\draw[edge] (4,0) -- (4,4);
\draw[edge] (6,0) -- (6,4);
\draw[edge] (8,0) -- (8,2);
\draw[red, line width=5pt] (2.6,5) -- (3.4,5) ;
\draw[red, line width=5pt] (6.6,3) -- (7.4,3) ;
\end{tikzpicture}%
\captionsetup{labelformat=empty}
\captionof{figure}{$(3,3,1)$}
\end{minipage}
\begin{minipage}{.18\textwidth}
\centering
\begin{tikzpicture}[x=\unitsize,y=\unitsize,baseline=0]
\tikzset{vertex/.style={}}%
\tikzset{edge/.style={very thick}}%
\draw[dotted] (0,0) -- (12,0);
\draw[dotted] (0,2) -- (12,2);
\draw[dotted] (0,4) -- (12,4);
\draw[dotted] (0,6) -- (12,6);
\draw[dotted] (2,-2) -- (2,8);
\draw[dotted] (4,-2) -- (4,8);
\draw[dotted] (6,-2) -- (6,8);
\draw[dotted] (8,-2) -- (8,8);
\draw[dotted] (10,-2) -- (10,8);
\draw[edge] (2,0) -- (8,0);
\draw[edge] (2,2) -- (8,2);
\draw[edge] (2,4) -- (6,4);
\draw[edge] (2,0) -- (2,4);
\draw[edge] (4,0) -- (4,4);
\draw[edge] (6,0) -- (6,4);
\draw[edge] (8,0) -- (8,2);
\draw[red, line width=5pt] (2.6,5) -- (3.4,5) ;
\draw[red, line width=5pt] (8.6,1) -- (9.4,1) ;
\end{tikzpicture}%
\captionsetup{labelformat=empty}
\captionof{figure}{$(4,2,1)$}
\end{minipage}
\begin{minipage}{.18\textwidth}
\centering
\begin{tikzpicture}[x=\unitsize,y=\unitsize,baseline=0]
\tikzset{vertex/.style={}}%
\tikzset{edge/.style={very thick}}%
\draw[dotted] (0,0) -- (12,0);
\draw[dotted] (0,2) -- (12,2);
\draw[dotted] (0,4) -- (12,4);
\draw[dotted] (0,6) -- (12,6);
\draw[dotted] (2,-2) -- (2,8);
\draw[dotted] (4,-2) -- (4,8);
\draw[dotted] (6,-2) -- (6,8);
\draw[dotted] (8,-2) -- (8,8);
\draw[dotted] (10,-2) -- (10,8);
\draw[edge] (2,0) -- (8,0);
\draw[edge] (2,2) -- (8,2);
\draw[edge] (2,4) -- (6,4);
\draw[edge] (2,0) -- (2,4);
\draw[edge] (4,0) -- (4,4);
\draw[edge] (6,0) -- (6,4);
\draw[edge] (8,0) -- (8,2);
\draw[red, line width=5pt] (2.6,5) -- (3.4,5) ;
\draw[red, line width=5pt] (6.6,3) -- (7.4,3) ;
\draw[red, line width=5pt] (8.6,1) -- (9.4,1) ;
\end{tikzpicture}%
\captionsetup{labelformat=empty}
\captionof{figure}{$(4,3,1)$}
\end{minipage}
\begin{minipage}{.2\textwidth}
\centering
\begin{tikzpicture}[x=\unitsize,y=\unitsize,baseline=0]
\tikzset{vertex/.style={}}%
\tikzset{edge/.style={very thick}}%
\draw[dotted] (0,0) -- (12,0);
\draw[dotted] (0,2) -- (12,2);
\draw[dotted] (0,4) -- (12,4);
\draw[dotted] (0,6) -- (12,6);
\draw[dotted] (2,-2) -- (2,8);
\draw[dotted] (4,-2) -- (4,8);
\draw[dotted] (6,-2) -- (6,8);
\draw[dotted] (8,-2) -- (8,8);
\draw[dotted] (10,-2) -- (10,8);
\draw[edge] (2,0) -- (8,0);
\draw[edge] (2,2) -- (8,2);
\draw[edge] (2,4) -- (6,4);
\draw[edge] (2,0) -- (2,4);
\draw[edge] (4,0) -- (4,4);
\draw[edge] (6,0) -- (6,4);
\draw[edge] (8,0) -- (8,2);
\draw[red, line width=5pt] (7,3) -- (9,3) -- (9,1) ;
\end{tikzpicture}%
\captionsetup{labelformat=empty}
\captionof{figure}{$(4,4)$}
\end{minipage}
\begin{minipage}{.2\textwidth}
\centering
\begin{tikzpicture}[x=\unitsize,y=\unitsize,baseline=0]
\tikzset{vertex/.style={}}%
\tikzset{edge/.style={very thick}}%
\draw[dotted] (0,0) -- (12,0);
\draw[dotted] (0,2) -- (12,2);
\draw[dotted] (0,4) -- (12,4);
\draw[dotted] (0,6) -- (12,6);
\draw[dotted] (2,-2) -- (2,8);
\draw[dotted] (4,-2) -- (4,8);
\draw[dotted] (6,-2) -- (6,8);
\draw[dotted] (8,-2) -- (8,8);
\draw[dotted] (10,-2) -- (10,8);
\draw[edge] (2,0) -- (8,0);
\draw[edge] (2,2) -- (8,2);
\draw[edge] (2,4) -- (6,4);
\draw[edge] (2,0) -- (2,4);
\draw[edge] (4,0) -- (4,4);
\draw[edge] (6,0) -- (6,4);
\draw[edge] (8,0) -- (8,2);
\draw[red, line width=5pt] (2.6,5) -- (3.4,5) ;
\draw[red, line width=5pt] (7,3) -- (9,3) -- (9,1) ;
\end{tikzpicture}%
\captionsetup{labelformat=empty}
\captionof{figure}{$(4,4,1)$}
\end{minipage}
\end{example}

\bigskip

\section{The linear strands are $\gl(m|n)$-modules}\label{sec:strands=glmn}

In this section we consider finitely generated $\GL$-equivariant (graded) $S$-modules $M$. We say that $M$ is \defi{symmetric} if it decomposes as
\begin{equation}\label{eq:M-symmetric}
M = \bigoplus_{\ll\in\bb{N}_{\dom}} (\SS_{\ll}V_0 \oo \SS_{\ll}V_1)^{\oplus m_{\ll}},\mbox{ where each }m_{\ll}\mbox{ is a non-negative integer}.
\end{equation}
We will always consider the natural grading on $M$ where $\SS_{\ll}V_0 \oo \SS_{\ll}V_1$ is placed in degree $|\ll|$. Examples of symmetric modules include all the $\GL$-equivariant ideals $I\subseteq S$, as well as quotients of such ideals. The goal of this section is to explain why for a symmetric module $M$, the linear strands of its minimal free resolution translate via the BGG correspondence to modules over the Lie superalgebra $\g=\gl(m|n)$: in fact, the modules that we get in this way have composition factors given by the simples $\LL_{\ll}$ defined in Section~\ref{subsec:glmn}.

\begin{theorem}\label{thm:strands=glmn}
 If $M$ is a symmetric $S$-module then for every $t\in\bb{Z}$ we have that $H_t(\tl{\bf R}(M))$ is a $\g$-module of finite length, with simple composition factors of the form $\LL_{\ll}$ with $\ll\in\bb{N}^n_{\dom}$.
\end{theorem}

To prove the theorem we will show that the complex $\tl{\bf R}(M)$ is in fact a complex of $\g$-modules, each of which has finite length and has composition factors of the form $\LL_{\ll}$. The key observation is the following.
\begin{lemma}\label{eq:g0-equiv-is-gequiv}
 Suppose that $\mu\in\bb{N}^n_{\dom}$ is a partition obtained from $\ll$ by adding a single box. We have that every $\g_0$-equivariant homomorphism of $E$-modules between $\K_{\mu}$ and $\K_{\ll}$ is also $\g$-equivariant (and vice-versa).
\end{lemma}

\begin{proof}
 Under the assumptions on $\ll$ and $\mu$, it follows from Pieri's rule that there exists a unique (up to scalar) $\g_0$-equivariant inclusion
\[ S_{\mu}W_0 \oo S_{\mu}W_1 \subset \K_{\ll},\]
and since $S_{\mu}W_0 \oo S_{\mu}W_1$ are the generators of the free $E$-module $\K_{\mu}$ we get
\begin{equation}\label{eq:Hom-g0-E=C}
 \bb{C}=\Hom_{\g_0}(S_{\mu}W_0 \oo S_{\mu}W_1,\K_{\ll}) = \Hom_{\g_0,E}(\K_{\mu},\K_{\ll}),
\end{equation}
where $\Hom_{\g_0,E}(\bullet,\bullet)$ is the $\Hom$-functor in the category of $\g_0$-equivariant $E$-modules.

The statement of our lemma is then equivalent to the fact that there exists a unique (up to scalar) non-zero homomorphism of $\g$-modules between $\K_{\mu}$ and $\K_{\ll}$. Since the highest weight vector of $S_{\mu}W_0 \oo S_{\mu}W_1$ is a \defi{primitive weight vector} in $\K_{\ll}$ (see \cite[Section~3.2]{su-zhang} for the terminology) it follows that the subspace $S_{\mu}W_0 \oo S_{\mu}W_1$ of $\K_{\ll}$ is annihilated by $\g_1$ and therefore it forms a $\p$-submodule of $\K_{\ll}$. We get that 
\begin{equation}\label{eq:Hom-g=C}
 \bb{C}=\Hom_{\g_0}(S_{\mu}W_0 \oo S_{\mu}W_1,\K_{\ll}) = \Hom_{\p}(S_{\mu}W_0 \oo S_{\mu}W_1,\K_{\ll}) = \Hom_{\g}(\K_{\mu},\K_{\ll}),
\end{equation}
where the last equality follows from Frobenius reciprocity. Combining (\ref{eq:Hom-g0-E=C}) with (\ref{eq:Hom-g=C}) yields the desired conclusion.
\end{proof}

\begin{proof}[Proof of Theorem~\ref{thm:strands=glmn}]
 Consider a symmetric module $M$ and note that via the BGG correspondence we have that $H_t(\tl{\bf R}(M))$ is the middle homology of the $3$-term complex
\begin{equation}\label{eq:Mvee-3term}
 M^{\vee}_{t+1} \oo E \lra M^{\vee}_t \oo E \lra M^{\vee}_{t-1} \oo E
\end{equation}
where the maps respect the $E$-module structure and are $\GL$-equivariant, and hence also $\g_0$-equivariant. Using (\ref{eq:M-symmetric}) we get that
\[ M^{\vee}_t \oo E = \bigoplus_{|\ll| = t} \K_{\ll}^{\oplus m_{\ll}}\]
so the maps in (\ref{eq:Mvee-3term}) are sums of $\g_0$-equivariant homomorphisms of $E$-modules between $\K_{\mu}$ and $\K_{\ll}$, where $\mu,\ll$ vary over pairs of partitions with $\mu$ obtained from $\ll$ by the addition of a single box. Such maps are by Lemma~\ref{eq:g0-equiv-is-gequiv} homomorphisms of $\g$-modules, so (\ref{eq:Mvee-3term}) is a complex of $\g$-modules, and hence the same is true about its cohomology. Since $M$ is finitely generated, its graded components are finite dimensional, so $M^{\vee}_t \oo E$ has finite length. Since each of the modules $\K_{\ll}$ has composition factors of the form $\LL_{\delta}$, the same must be true about $H_t(\tl{\bf R}(M))$.
\end{proof}

\section{The main conjecture}\label{sec:conjecture}

 For $\ll$ a partition with at most $n$ parts,  we consider the set of $\ll$-admissible augmented Dyck patterns 
 \[\bb{D}=(D_1,\cdots,D_r;\bb{B})\]
 with no Dyck path of length one, and for which $\ll(\bb{D})$ has at most $n$ parts:
\begin{equation}\label{eq:A-ll-n}
\mc{A}(\ll;n) = \{ \bb{D}\mbox{ a }\ll\mbox{-admissible Dyck pattern}: |D_i|\geq 3\mbox{ for all }i=1,\cdots,r,\mbox{ and }\ll(\bb{D})_j = 0\mbox{ for }j>n\}.
\end{equation}

\begin{conjecture}\label{conj:main}
 Suppose that $m\geq n$ are positive integers, $S$ is the coordinate ring of $\bb{C}^{m\times n}$, $\ll$ is a partition with at most $n$ parts, and $I_{\ll}\subseteq S$ is the corresponding $\GL$-equivariant ideal. For $b\geq 0$ we have the following equality in the Grothendieck group $K_0(\g)$ of finite dimensional representations of $\g$.
\begin{equation}\label{eq:conj-main}
  \left[H_{|\ll|+b}(\tl{\bf R}(I_{\ll})) \right]= \sum_{\substack{\bb{D}\in\mc{A}(\ll;n) \\ b(\bb{D}) = b}} [\bb{L}_{\ll(\bb{D})}].
\end{equation}
\end{conjecture}

Since $I_{\ll}$ has no generators of degree smaller than $|\ll|$ it follows that $H_{|\ll|+b}(\tl{\bf R}(I_{\ll}))=0$ for $b<0$, which is why Conjecture~\ref{conj:main} is restricted to $b\geq 0$. The Castelnuovo--Mumford regularity of $I_{\ll}$ is the maximal value of $r$ for which $H_{r}(\tl{\bf R}(I_{\ll}))\neq 0$ and will be discussed in Section~\ref{subsec:regularity}. We think of each of the $\g$-modules $\bb{L}_{\ll(\bb{D})}$ as giving rise via the BGG correspondence to a linear complex appearing as a subquotient in the minimal free resolution of $I_{\ll}$. As such, $b(\bb{D})$ is measuring the vertical displacement of $\bb{L}_{\ll(\bb{D})}$ within the Betti table, while $d(\bb{D})$ is measuring its horizontal displacement. More precisely, $\bb{L}_{\ll(\bb{D})}$ corresponds to a linear complex that appears entirely within the row indexed by $|\ll| + b(\bb{D})$ of the Betti table, starting in column $d(\bb{D})$. 

\begin{example}\label{ex:BIlambda} Consider $m=n=3$ and $\ll=(3,2)$. The conjecture asserts that
\begin{equation}\label{eq:equiv-b32}
\begin{aligned}
\left[H_{5}(\tl{\bf R}(I_{(3,2)})) \right] &= [\LL_{(3,2)}] + [\LL_{(4,4)}]\\
\left[H_{6}(\tl{\bf R}(I_{(3,2)})) \right] &= [\LL_{(3,3,3)}] + [\LL_{(4,4,3)}]\\
\left[H_{7}(\tl{\bf R}(I_{(3,2)})) \right] &= [\LL_{(5,5,5)}]
\end{aligned}
\end{equation}
since the Dyck patterns in $\mc{A}((3,2);3)$ are as follows (labelled by $\ll(\bb{D})$):

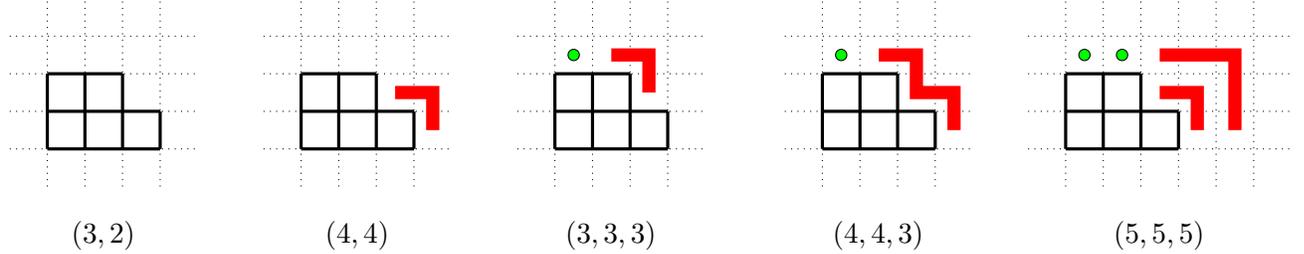
\begin{minipage}{.18\textwidth}
\centering
\begin{tikzpicture}[x=\unitsize,y=\unitsize,baseline=0]
\tikzset{vertex/.style={}}%
\tikzset{edge/.style={very thick}}%
\draw[dotted] (0,0) -- (10,0);
\draw[dotted] (0,2) -- (10,2);
\draw[dotted] (0,4) -- (10,4);
\draw[dotted] (0,6) -- (10,6);
\draw[dotted] (2,-2) -- (2,8);
\draw[dotted] (4,-2) -- (4,8);
\draw[dotted] (6,-2) -- (6,8);
\draw[dotted] (8,-2) -- (8,8);
\draw[edge] (2,0) -- (8,0);
\draw[edge] (2,2) -- (8,2);
\draw[edge] (2,4) -- (6,4);
\draw[edge] (2,0) -- (2,4);
\draw[edge] (4,0) -- (4,4);
\draw[edge] (6,0) -- (6,4);
\draw[edge] (8,0) -- (8,2);
\end{tikzpicture}%
\captionsetup{labelformat=empty}
\captionof{figure}{$(3,2)$}
\end{minipage}
\begin{minipage}{.18\textwidth}
\centering
\begin{tikzpicture}[x=\unitsize,y=\unitsize,baseline=0]
\tikzset{vertex/.style={}}%
\tikzset{edge/.style={very thick}}%
\draw[dotted] (0,0) -- (10,0);
\draw[dotted] (0,2) -- (10,2);
\draw[dotted] (0,4) -- (10,4);
\draw[dotted] (0,6) -- (10,6);
\draw[dotted] (2,-2) -- (2,8);
\draw[dotted] (4,-2) -- (4,8);
\draw[dotted] (6,-2) -- (6,8);
\draw[dotted] (8,-2) -- (8,8);
\draw[edge] (2,0) -- (8,0);
\draw[edge] (2,2) -- (8,2);
\draw[edge] (2,4) -- (6,4);
\draw[edge] (2,0) -- (2,4);
\draw[edge] (4,0) -- (4,4);
\draw[edge] (6,0) -- (6,4);
\draw[edge] (8,0) -- (8,2);
\draw[red, line width=5pt] (7,3) -- (9,3) -- (9,1) ;
\end{tikzpicture}%
\captionsetup{labelformat=empty}
\captionof{figure}{$(4,4)$}
\end{minipage}
\begin{minipage}{.18\textwidth}
\centering
\begin{tikzpicture}[x=\unitsize,y=\unitsize,baseline=0]
\tikzset{vertex/.style={}}%
\tikzset{edge/.style={very thick}}%
\draw[dotted] (0,0) -- (10,0);
\draw[dotted] (0,2) -- (10,2);
\draw[dotted] (0,4) -- (10,4);
\draw[dotted] (0,6) -- (10,6);
\draw[dotted] (2,-2) -- (2,8);
\draw[dotted] (4,-2) -- (4,8);
\draw[dotted] (6,-2) -- (6,8);
\draw[dotted] (8,-2) -- (8,8);
\draw[edge] (2,0) -- (8,0);
\draw[edge] (2,2) -- (8,2);
\draw[edge] (2,4) -- (6,4);
\draw[edge] (2,0) -- (2,4);
\draw[edge] (4,0) -- (4,4);
\draw[edge] (6,0) -- (6,4);
\draw[edge] (8,0) -- (8,2);
\draw[red, line width=5pt] (5,5) -- (7,5) -- (7,3);
\draw[fill=green] (3,5) circle [radius=0.3] ;
\end{tikzpicture}%
\captionsetup{labelformat=empty}
\captionof{figure}{$(3,3,3)$}
\end{minipage}
\begin{minipage}{.2\textwidth}
\centering
\begin{tikzpicture}[x=\unitsize,y=\unitsize,baseline=0]
\tikzset{vertex/.style={}}%
\tikzset{edge/.style={very thick}}%
\draw[dotted] (0,0) -- (10,0);
\draw[dotted] (0,2) -- (10,2);
\draw[dotted] (0,4) -- (10,4);
\draw[dotted] (0,6) -- (10,6);
\draw[dotted] (2,-2) -- (2,8);
\draw[dotted] (4,-2) -- (4,8);
\draw[dotted] (6,-2) -- (6,8);
\draw[dotted] (8,-2) -- (8,8);
\draw[edge] (2,0) -- (8,0);
\draw[edge] (2,2) -- (8,2);
\draw[edge] (2,4) -- (6,4);
\draw[edge] (2,0) -- (2,4);
\draw[edge] (4,0) -- (4,4);
\draw[edge] (6,0) -- (6,4);
\draw[edge] (8,0) -- (8,2);
\draw[red, line width=5pt] (5,5) -- (7,5) -- (7,3) -- (9,3) -- (9,1) ;
\draw[fill=green] (3,5) circle [radius=0.3] ;
\end{tikzpicture}%
\captionsetup{labelformat=empty}
\captionof{figure}{$(4,4,3)$}
\end{minipage}
\begin{minipage}{.2\textwidth}
\centering
\begin{tikzpicture}[x=\unitsize,y=\unitsize,baseline=0]
\tikzset{vertex/.style={}}%
\tikzset{edge/.style={very thick}}%
\draw[dotted] (0,0) -- (14,0);
\draw[dotted] (0,2) -- (14,2);
\draw[dotted] (0,4) -- (14,4);
\draw[dotted] (0,6) -- (14,6);
\draw[dotted] (2,-2) -- (2,8);
\draw[dotted] (4,-2) -- (4,8);
\draw[dotted] (6,-2) -- (6,8);
\draw[dotted] (8,-2) -- (8,8);
\draw[dotted] (10,-2) -- (10,8);
\draw[dotted] (12,-2) -- (12,8);
\draw[edge] (2,0) -- (8,0);
\draw[edge] (2,2) -- (8,2);
\draw[edge] (2,4) -- (6,4);
\draw[edge] (2,0) -- (2,4);
\draw[edge] (4,0) -- (4,4);
\draw[edge] (6,0) -- (6,4);
\draw[edge] (8,0) -- (8,2);
\draw[red, line width=5pt] (7,3) -- (9,3) -- (9,1);
\draw[red, line width=5pt] (7,5) -- (11,5) -- (11,1);
\draw[fill=green] (3,5) circle [radius=0.3] ;
\draw[fill=green] (5,5) circle [radius=0.3] ;
\end{tikzpicture}%
\captionsetup{labelformat=empty}
\captionof{figure}{$(5,5,5)$}
\end{minipage}

\end{example}

\bigskip

The Betti table of $I_{(3,2)}$ computed using Macaulay2 \cite{M2} is as follows (recall the convention that the Betti number $\b_{i,i+j} = \dim_{\bb{C}}\Tor^S_i(I_{(3,2)},S)_{i+j}$ is placed in row $j$, column $i$):
\begin{equation}\label{eq:betti-I32}
\begin{matrix}
     &0&1&2&3&4&5&6&7&8\\
     \text{5:}&225&
     1132&2673&3807&3485&2016&675&100&\text{.}\\\text{6:}&\text{.}&\text{.}&\text{.}&1&\text{.}&
     9&16&9&\text{.}\\\text{7:}&\text{.}&\text{.}&\text{.}&\text{.}&\text{.}&\text{.}&\text{.}&\text{.}&1\\
\end{matrix}
\end{equation}
The reader may now reconcile (\ref{eq:equiv-b32}) with (\ref{eq:betti-I32}) based on the following Hilbert series calculations, which can be obtained starting with the Hilbert series of Kac modules by inverting the relationship between simple and Kac modules in Theorem~\ref{thm:compos-Kac} (see for instance \cite[Section~4]{su-zhang-char} for the general case): if we write $\HS_{\mu}(t)$ for the Hilbert series of the graded $E$-module $\LL_{\mu}$ then
\[\HS_{(3,2)}(t) = 225t^5 + 1132t^6 + 2673t^7 + 3582t^8 + 2785t^9 + 1188t^{10} + 225 t^{11},\]
\[\HS_{(4,4)}(t) = 225t^8 + 700t^9 + 828t^{10} + 450t^{11} + 100 t^{12},\]
\[\HS_{(3,3,3)}(t) = t^9,\quad\HS_{(4,4,3)}(t) = 9t^{11}+16t^{12}+9t^{13},\mbox{ and }\HS_{(5,5,5)}(t) = t^{15}.\]

\section{Some evidence in support of the main conjecture}\label{sec:evidence}

The goal of this section is to illustrate some results that provide supporting evidence for Conjecture~\ref{conj:main}. In Section~\ref{subsec:linear-strand} we prove that the said conjecture predicts correctly the structure of the first linear strand in the minimal free resolution of any ideal $I_{\ll}$. If true, Conjecture~\ref{conj:main} would imply a formula for the Castelnuovo--Mumford regularity of any ideal $I_{\ll}$; we explain in Section~\ref{subsec:regularity} how this formula is equivalent to the one proved in \cite[Theorem~5.1]{raicu-weyman}. Finally, in Section~\ref{subsec:rectangular} we consider the ideals $I_{\ll}$ when $\ll$ is a rectangular partition: we prove that Conjecture~\ref{conj:main} holds in this case, by showing that it is equivalent to \cite[Theorem~3.1]{raicu-weyman-syzygies}.

\subsection{The first linear strand}\label{subsec:linear-strand}

 We consider a partition $\ll\in\bb{N}^n_{\dom}$ and let $\mc{A}^{\circ}(\ll;n)\subseteq\mc{A}(\ll;n)$ be the subset consisting of the Dyck patterns with no bullets:
 \[\mc{A}^{\circ}(\ll;n) = \{\bb{D}=(D_1,\cdots,D_r)\mbox{ with }\bb{D}\in\mc{A}(\ll;n)\}.\]
The goal of this section is to prove the following theorem, which is the case $b=0$ in Conjecture~\ref{conj:main}.
\begin{theorem}\label{thm:linear-strand}
 If we let $d=|\ll|$ then we have the following equality in $K_0(\g)$:
 \[[H_d(\tl{\bf R}(I_{\ll}))] = \sum_{\bb{D}\in\mc{A}^{\circ}(\ll;n)} [\LL_{\ll(\bb{D})}].\]
\end{theorem}

\begin{example}\label{ex:1st-strand}
 Consider again the case when $n=3$ and $\ll=(3,2)$. The only Dyck patterns in Example~\ref{ex:BIlambda} that contain no bullets are the ones for which $\ll(\bb{D})$ is $(3,2)$ or $(4,4)$, and as we have seen they are precisely the ones contributing to the first linear strand of the Betti table (\ref{eq:betti-I32}). These patterns are also the only patterns in Example~\ref{ex:compos-Kac} that contain no Dyck paths of length one.
\end{example}

\begin{proof}[Proof of Theorem~\ref{thm:linear-strand}]
 Using the BGG correspondence as described in Section~\ref{subsec:BGG} we get that
 \[H_d(\tl{\bf R}(I_{\ll})) = \coker\left(\bigoplus_{\substack{|\mu|=d+1 \\ \mu\geq\ll}}\K_{\mu} \overset{\psi}{\lra} \K_{\ll} \right).\]
 Since each $\K_{\mu}$ is generated by a $\g_0$-highest weight vector in $\SS_{\mu}W_0\oo\SS_{\mu}W_1$, the image of $\psi$ is the submodule of $\K_{\ll}$ generated by the primitive weight vectors of weight $\mu$, as $\mu$ ranges over the partitions of $d+1$ with $\mu\geq\ll$. If we interpret the composition factors of $\K_{\ll}$ as in Theorem~\ref{thm:compos-Kac}, it follows from \cite[Theorem~5.18]{su-zhang} that the composition factors of $\operatorname{Im}\psi$ are the modules $\LL_{\ll(\bb{D})}$ where $\bb{D}=(D_1,\cdots,D_r)$ range over patterns $\bb{D}\in\mc{K}(\ll;n)$ satisfying the condition $|D_i|=1$ for some~$i$. Since $\mc{A}^{\circ}(\ll;n)$ has an equivalent description as
 \[\mc{A}^{\circ}(\ll;n) = \{\bb{D}=(D_1,\cdots,D_r)\in\mc{K}(\ll;n) : |D_i|\geq 3\mbox{ for all }i=1,\cdots,r\}\]
it follows that $\coker(\psi)$ has composition factors $\LL_{\ll(\bb{D})}$ where $\bb{D}\in\mc{A}^{\circ}(\ll;n)$, as desired.
\end{proof}

\subsection{Castelnuovo--Mumford regularity}\label{subsec:regularity}

Recalling that the bullet size $b(\bb{D})$ measures the vertical displacement of $\LL_{\ll(\bb{D})}$ within the Betti table of $I_{\ll}$, we see that Conjecture~\ref{conj:main} implicitly describes the Castelnuovo--Mumford regularity of an ideal $I_{\ll}$ as
\begin{equation}\label{eq:regularity}
 \reg(I_{\ll}) = |\ll| + \max_{\bb{D}\in\mc{A}(\ll;n)} b(\bb{D}).
\end{equation}
In \cite[Theorem~5.1]{raicu-weyman} we proved (with the convention that $\ll_{n+1}=-1$) that
\begin{equation}\label{eq:proved-reg}
\reg(I_{\ll})=\max_{\substack{p=1,\cdots,n \\ \ll_{p}>\ll_{p+1}}}(n\cdot \ll_p+(p-2)\cdot(n-p)).
\end{equation}
In what follows we show that (\ref{eq:regularity}) and (\ref{eq:proved-reg}) are equivalent, so the prediction (\ref{eq:regularity}) is indeed accurate.

Suppose first that $1\leq p\leq n$ is such that $\ll_p>\ll_{p+1}$. We will construct a pattern $\bb{D}\in\mc{A}(\ll;n)$ satisfying 
\[|\ll| + b(\bb{D})\geq n\cdot \ll_p+(p-2)\cdot(n-p)\]
proving that the quantity in (\ref{eq:regularity}) is greater than or equal to that in (\ref{eq:proved-reg}). We construct $\bb{D}$ by considering a succession of hooks of minimal length ($3,5,7,\cdots$) around the corner $(\ll_p,p)$ of the partition $\ll$. More precisely, we consider the Dyck pattern $\bb{D}$ defined as $\bb{D}=(D_1,D_2,\cdots,D_{n-p};\bb{B})$, where
\begin{equation}\label{eq:corner-hooks}
D_i = \{(\ll_p+i-j,p+i) : j=0,\cdots,i\} \cup \{(\ll_p+i,p+i-j) : j=1,\cdots,i\},
\end{equation}
and $\bb{B}$ (or equivalently $\ll(\bb{D})$) are as determined in Lemma~\ref{lem:adm-pat-determines-bullets} by $\ll$ and the paths $D_i$. We have that
\[d(\bb{D}) = \sum_{i=1}^{n-p}(2i+1) = (n-p+1)^2 - 1 = (n-p+2)(n-p).\]
Since $(\ll_p+n-p,n)\in D_{n-p} \subseteq \ll(\bb{D})$, it follows that $\ll(\bb{D})$ contains the rectangular partition $n\times(\ll_p+n-p)$, so that $|\ll(\bb{D})| \geq n\cdot (\ll_p + n - p)$ and therefore
\[|\ll| + b(\bb{D}) = |\ll(\bb{D})| - d(\bb{D}) \geq n\cdot (\ll_p + n - p) - (n-p+2)(n-p) = n\cdot \ll_p + (p-2)\cdot(n-p).\]


Suppose now that $\bb{D}_{\max}\in\mc{A}(\ll;n)$ is a pattern that maximizes (\ref{eq:regularity}). We prove that the quantity in (\ref{eq:proved-reg}) is greater than or equal to that in (\ref{eq:regularity}) by finding a value of $p$ for which $\ll_p>\ll_{p+1}$ and
\begin{equation}\label{eq:52>=51}
n\cdot \ll_p+(p-2)\cdot(n-p)\geq |\ll| + b(\bb{D}_{\max}).
\end{equation}
This is sufficient to conclude the equivalence between (\ref{eq:regularity}) and (\ref{eq:proved-reg}). We consider the right-most corner of the partition $\ll(\bb{D}_{\max})$ which is not a corner of $\ll$: it has coordinates $(\ll_i+r,i+r)$ for some $1\leq i\leq n$ and $1\leq r\leq n-i$. Since $(\ll_i+1,i+1)$ belongs to some Dyck path $D_i$ in $\bb{D}_{\max}$, it follows that there are no bullets $(x,y)\in\bb{B}$ with $x\geq\ll_i$ and $y\geq i$. Since $(\ll_i+r,i+r)$ is the rightmost corner of $\ll(\bb{D}_{\max})$ not in $\ll$, it follows that no bullets $(x,y)\in\bb{B}$ have $x>\ll_i + r$, and since $r\leq n-i$ we get that no bullets have $x>\ll_i+n-i$. It follows that
\begin{equation}\label{eq:ineq-Dmax}
|\ll| + b(\bb{D}_{\max}) \leq \ll_i + (n-i)\cdot(\ll_i-1) + \sum_{j=1}^{i-1}\max(\ll_j,\ll_i+n-i).
\end{equation}
If $\ll_j\leq \ll_i + (n-i)$ for all $j<i$ then the above inequality becomes
\[|\ll| + b(\bb{D}_{\max}) \leq \ll_i + (n-i)\cdot(\ll_i-1) + (i-1)\cdot(\ll_i+n-i) = n\cdot \ll_i + (i-2)\cdot(n-i),\]
so we may choose $p=i$ in (\ref{eq:52>=51}). Otherwise, let $1\leq p\leq i-1$ be such that
\[ \ll_p > \ll_i+(n-i) \geq \ll_{p+1}\]
and consider the pattern $\bb{D}$ with Dyck paths given by (\ref{eq:corner-hooks}). We have that
\[
\begin{aligned}
|\ll| + b(\bb{D}) &= (n-p)\cdot(\ll_p-1) + \ll_p + \sum_{j=1}^{p-1}\max(\ll_j,\ll_p+n-p) \\
&> (n-p+1)\cdot(\ll_p-1) + \sum_{j=1}^{p-1}\max(\ll_j,\ll_i+n-i) \\
&\geq (n-p+1)\cdot(\ll_i+n-i) + \sum_{j=1}^{p-1}\max(\ll_j,\ll_i+n-i).
\end{aligned}
\]
To contradict the maximality of $\bb{D}_{\max}$ using (\ref{eq:ineq-Dmax}) it is then enough to check that
\[
\begin{aligned}
(n-p+1)\cdot(\ll_i+n-i) &\geq \ll_i + (n-i)\cdot(\ll_i-1) + \sum_{j=p}^{i-1}\max(\ll_j,\ll_i+n-i) \\
&= \ll_i + (n-i)\cdot(\ll_i-1) + (i-p)\cdot(\ll_i+n-i) \\
&= \ll_i\cdot(n-p+1) + (i-p-1)\cdot(n-i).
\end{aligned}
\]
This inequality can be rewritten as
\[(n-p+1)\cdot(n-i)\geq(i-p-1)\cdot(n-i)\]
which holds because $n+1\geq i-1$.

\subsection{Rectangular ideals}\label{subsec:rectangular}

In this section we show that when $\ll = a\times b$ is a rectangular partition, the conclusion of Conjecture~\ref{conj:main} coincides with the main theorem in \cite{raicu-weyman-syzygies}, and is therefore correct. To do so, we will consider a different encoding of the information in (\ref{eq:conj-main}), as follows. We introduce a variable $w$ that keeps track of cohomological shifts, and define the \defi{$\g$-equivariant Betti polynomial} of $I_{\ll}$ to be $B_{\ll}(w)\in K_0(\g)[w]$ defined by
\[B_{\ll}(w) = \sum_{\substack{j\in\bb{Z} \\ \mu\in\bb{N}^n_{\dom}}} m_{\mu,j} \cdot [\LL_{\mu}]\cdot w^j\]
where $m_{\mu,j}$ are non-negative multiplicities uniquely determined by the equalities
\[\left[H_{t}(\tl{\bf R}(I_{\ll})) \right] = \sum_{\mu\in\bb{N}^n_{\dom}} m_{\mu,|\mu| - t} \cdot [\LL_{\mu}] \mbox{ for all }t\in\bb{Z}.\]
The reason for these conventions is that since $\LL_{\mu}$ is generated in degree $|\mu|$, if the corresponding linear complex of $S$-modules lies within the row indexed by $t$ of the Betti table of $I_{\ll}$ then its initial term must be located in cohomological degree $|\mu|-t$ (see Example~\ref{ex:BIlambda}). Recalling the observation that $d(\bb{D})$ measures the horizontal displacement of $\LL_{\ll(\bb{D})}$ in the Betti table (i.e. the cohomological shift of the corresponding linear complex) we can rewrite (\ref{eq:conj-main}) as
\begin{equation}\label{eq:new-conj-main}
  B_{\ll}(w) = \sum_{\bb{D}\in\mc{A}(\ll;n)} \bb{L}_{\ll(\bb{D})} \cdot w^{d(\bb{D})}.
\end{equation}

Returning to the case when $\ll = a\times b$ is a rectangular partition, we note that the Dyck paths in a $\ll$-admissible Dyck pattern are hooks centered at $(b+i,a+i)$, whose length increases with $i$. More precisely, if $\bb{D}\in\mc{A}(\ll;n)$ then $\bb{D} = (D_1,D_2,\cdots,D_q;\bb{B})$ where $q\leq n-a$, $|D_i| = 2l_i+1$ for some positive integers $l_i$ satisfying
\[1\leq l_1<l_2<\cdots<l_q\leq \min(a,b) + q -1,\]
and each $D_i$ is a hook consisting of the boxes
\[D_i = \{(b+i-j,a+i) : j=0,\cdots,l_i\} \cup \{(b+i,a+i-j) : j=1,\cdots,l_i\},\]
whereas $\bb{B}$ is determined by the fact that $\ll(\bb{D}) = (a+q)\times(b+q)$ (see Lemma~\ref{lem:adm-pat-determines-bullets}). If we make the change of variable $t_i = l_i - i$ for $i=1,\cdots,q$, we see that
\[0\leq t_1\leq\cdots\leq t_q\leq \min(a,b)-1,\]
that is the patterns $\bb{D}\in\mc{A}(\ll;n)$ for which $\ll(\bb{D}) = (a+q)\times (b+q)$ are indexed by partitions $\ul{t}$ contained inside the $q\times(\min(a,b)-1)$ rectangular partition, and moreover their corresponding Dyck size is computed~by
\[ d(\bb{D}) = \sum_{i=1}^q (2l_i + 1) = \sum_{i=1}^q (2t_i + 2i + 1) = q^2 + 2q + 2|\ul{t}|.\]
Using the notation for Gauss polynomials from \cite[(1.7)]{raicu-weyman-syzygies} we have
\[ \sum_{0\leq t_1\leq\cdots\leq t_q\leq \min(a,b)-1} w^{2|\ul{t}|} = {q+\min(a,b)-1\choose q}_{w^2},\]
which shows that (\ref{eq:new-conj-main}) specializes to the following formula, equivalent to that of \cite[Theorem~3.1]{raicu-weyman-syzygies}:
\[B_{a\times b}(w) = \sum_{q=0}^{n-a} \LL_{(a+q)\times(b+q)} \cdot w^{q^2+2q} \cdot {q+\min(a,b)-1\choose q}_{w^2}.\] 

\section*{Acknowledgements}
Experiments with the computer algebra software Macaulay2 \cite{M2} have provided numerous valuable insights. Raicu acknowledges the support of the Alfred P. Sloan Foundation, and of the National Science Foundation Grant No.~1600765. Weyman acknowledges partial support of the Sidney Professorial Fund and of the National Science Foundation grant No.~1400740.

\end{document}